\documentclass[a4paper,12pt,reqno]{amsart}
\usepackage{amsfonts}
\usepackage{mathrsfs}
\usepackage{amsmath,amssymb,latexsym,amsfonts,amscd, yfonts}
\title{The Noether inequality for Gorenstein minimal 3-folds}
\author{Jungkai A. Chen and Meng Chen}

\address{\rm National Center for Theoretical Sciences, Taipei Office, and Department of Mathematics, National Taiwan University, Taipei, 106, Taiwan}
\email{jkchen@math.ntu.edu.tw}

\address{\rm Department of Mathematics \& LMNS, Fudan University, Shanghai, 200433, People's Republic of China}
\email{mchen@fudan.edu.cn}

\thanks{The first author was partially supported by
National Science Council of Taiwan and NCTS. The second author was supported by National Natural Science Foundation of China
(\#11171068, \#11121101, \#11231003) and Doctoral Fund of Ministry of Education of China (\#20110071110003)}

\newcommand{\bC}{{\mathbb C}}

\newcommand\lrw{\longrightarrow}

\newtheorem{thm}{Theorem}[section]

\newtheorem{cor}[thm]{Corollary}

\newtheorem{op}[thm]{Open problem}
\theoremstyle{definition}
\newtheorem{defn}[thm]{Definition}

\theoremstyle{remark}

\begin{document}
\numberwithin{equation}{section}
\begin{abstract} We prove the Conjecture  of Catenese--Chen--Zhang: the inequality
$K_X^3\geq \frac{4}{3}p_g(X)-\frac{10}{3}$ holds for all projective Gorenstein minimal 3-folds $X$ of general type.
\end{abstract}
\maketitle

\pagestyle{myheadings} \markboth{\hfill J. A. Chen and M. Chen
\hfill}{\hfill The Noether inequality\hfill}
%\tableofcontents

\section{\bf Introduction}
In the classification theory of algebraic varieties, the Noether
inequality, which asserts that $K^2 \ge 2p_g-4$ for minimal
surfaces of general type, plays a pivotal role. It is thus natural
and important to explore the higher dimensional analogue.

There are several attempts toward this direction. A naive guess is that, for minimal variety $X$ of general type,
$K_X^{\dim X} \ge 2( p_g(X) -\dim X)$, which holds in dimension $1$ and $2$. However, Kobayashi \cite{Kob} constructed  examples of canonically polarized threefolds with $p_g(X)=3k+4$ and $K_X^3=4k+2$ for $k
\ge 1$. Hence the inequality $K_X^3 \ge 2 p_g(X) -6$ fails in dimension 3 and one can only expect that $K_X^3 \ge \frac{4}{3} p_g(X) -\frac{10}{3}$.

The aim of this paper is to confirm the conjecture (\cite[Conj. 4.4]{CCZ}, in 2006) of Catanese--Chen--Zhang and to prove the following:

\begin{thm}\label{main} The inequality
$$K_X^3\geq \frac{4}{3}p_g(X)-\frac{10}{3}$$ holds for all projective Gorenstein minimal 3-folds $X$ of general type.
\end{thm}

Theorem \ref{main} was proved by the second author \cite{MRL} when
$X$ is canonically polarized and by Catanese--Chen--Zhang
\cite{CCZ} while $X$ is smooth minimal. We refer to the relevant
work \cite{Kob, MRL, JMSJ, CCZ} for more details of the history of this topic.

The main obstacle in proving the above theorem is the existence of
Gorenstein terminal singularities in the base
locus of the canonical linear system $|K_X|$, while $X$ is
canonically fibred by a family of curves of genus 2. By using
certain conceived and explicit resolution of Gorenstein terminal
singularities, which we call {\it feasible Goresntein resolution},
we are able to resolve the base locus and prove the statement.

Throughout we work over the complex number field $\bC$.

\section{\bf Special resolutions to Gorenstein terminal singularities $(X,P)$, pairs $(X,D)$ and linear systems $(X,|M|$)}

%Let $X_0$ be a projective Gorenstein threefold with terminal singularities and $P \in X_0$ a singularity.
%Assume that $|M|$ is a moving linear system on $X_0$ with $M$ a Cartier divisor and that $P\in \text{Bs}|M|$.

{}First of all, we recall the following result of the first author:

\begin{thm}\label{jk} (\cite[Theorem 1.3]{JK}) Let ${X}$ be an algebraic 3-fold with at worst terminal singularities. For any terminal singularity $P \in {X}$,
there exists a sequence of birational morphisms:
$$ \tau_P: Y=X_m \to  X_{m-1} \to \ldots  \to X_1 \to X_0={X},$$
such that $Y$ is smooth on $\tau_P^{-1}(P )$ and, for all $i$,  the morphism $\pi_i:
 X_{i+1} \to X_i$ is a divisorial contraction to a singular point $P_i\in X_i$ of index $r_i \ge 1$ with discrepancy $1/r_i$.
\end{thm}

Indeed, given a Gorenstein terminal singularity $(P \in {X})$, the
resolution can be constructed in explicit as follows.

\begin{enumerate}
\item  Take a divisorial contraction $\pi_1: X_1 \to {X}$
contracting $E_1$ to the point $P$ with discrepancy $1$, i.e.
$K_{X_1}=\pi_1^* (K_{{X}}) +E_1$.

\item
 If there are some higher index points on $E_1 \subset X_1$, there exists a Gorenstein partial resolution
$$ X_{n_1} \to X_{n_1-1} \to \ldots \to X_2 \to X_1$$ such that,
\begin{itemize}
\item for any $j>0$, the birational morphism $\pi_{j+1}: X_{j+1}
\to X_j$ is a divisorial contraction to a point $P_j \in X_j$ of
index $r_j>1$ with discrepancy $\frac{1}{r_j}$; \item $X_{n_1}$
has only Gorenstein terminal singularities of which each one is
``milder'' than $P \in {X}$.
\end{itemize}

\item
Inductively, we have a sequence of birational morphisms
$$\tau_P \colon  Y=X_{n_l} \to X_{n_{l-1}} \to \ldots \to X_{n_1} \to {X},$$
such that the birational morphism $\tau_{j+1}: X_{n_{j+1}} \to
X_{n_j}$ is constructed parallel to those in Steps (1) and (2),
$X_{n_{j+1}}$ has only Gorenstein terminal singularities and $Y$
is non-singular on $\tau_P^{-1}( P)$.
\end{enumerate}
\medskip

\begin{defn}
Given a Gorenstein terminal singularity $P \in {X}$, the
birational map $X_{n_1} \to {X} \ni P$ constructed as in Steps (1)
and (2) is called a {\it feasible Gorenstein partial resolution of
$P \in {X}$}, or {\it fG partial resolution} for short. The
birational morphism $ \tau_P: Y \to {X} \ni P$ constructed as in
Step (3) is called a {\it feasible resolution of $P \in {X}$}.
Clearly, $X_{n_{j+1}}$ is a fG partial resolution of $X_{n_j}$ for
any $j>0$.
\end{defn}

Now given a Gorenstein projective 3-fold  $X$  with terminal singularities.  Let $P \in X$ be a singular point and $D$ be an effective Cartier divisor on $X$ with $P\in D$.
We may consider a fG partial resolution of $P \in X$, say
\begin{equation} Z:=X_{n} \to \ldots \to X_1 \to  X_0=X, \label{pr}
\end{equation}
so that  the birational morphism $\pi_P: Z \to X$ is composed of a sequence of divisorial contractions $X_{i+1} \to X_i$ to  points $P_i \in X_i$ of index $r_i > 1$ with discrepancy $1/r_i$ for all $i >1$ together with a divisorial contraction $X_1 \to X$ to $P\in X$ with discrepancy $1$. Clearly, $Z$ is still a projective Gorenstein 3-fold with at worst terminal singularities.

For any $i>0$, let $D_i$ be the proper transform of $D$ in $X_i$
and write $D_{Z/X}:=\pi_P^*(D)-D_n$. Similarly, let $K_i$ be the
canonical divisor of $X_i$ and write $K_{Z/X}:=K_Z-\pi_P^*(K_X)$.
Also let $E_i$ be the exceptional divisor of the contraction
morphism $X_i \to X_{i-1}$ and $E_{i, X_j}$ denote the proper
transform of $E_i$ on $X_j$.

\begin{thm}\label{P} Given a projective Gorenstein 3-fold  $X$  with terminal singularities.  Let $P \in X$ be a singular point and $D$ be an effective Cartier divisor on $X$ with $P\in D$. Let $\pi_P: Z \to X$ be the fG partial resolution as in (\ref{pr}). Then  $D_{Z/X} \ge K_{Z/X}$.
\end{thm}

\begin{proof} %We have the partial resolution $\pi_P$ as in (\ref{pr}).
{}First of all, we have $K_{X_1/X}=E_1$ and $D_{X_1/X}=b_1E_1$,
where $b_1=\text{mult}_P D \in \mathbb{Z}_{>0}$ is the
multiplicity. Clearly, we have $D_{X_1/X} \ge K_{X_1/X}$.

Suppose we have $D_{X_i/X} \ge K_{X_i/X}$.
Write $K_{X_i/X}=\sum_{j=1}^i a_j E_j$, and $D_{X_i/X}=\sum_{j=1}^i b_j E_j$ with $b_j \ge a_j \in \mathbb{Z}$ for all $j$.
Since $\pi_i: X_{i+1} \to X_i$ is a divisorial contraction to a point $P_i$ of index $r>1$ with discrepancy $1/r$. Let
\begin{eqnarray} \pi_i^* (E_{j,X_i}) &=&E_{j, X_{i+1}}+\frac{\alpha_{i,j}}{r} E_{i+1}; \nonumber\\ \pi_i^*(D_i)&=&D_{i+1}+\frac{\beta_i}{r} E_{i+1} \end{eqnarray}
where $\alpha_{i,j} \ge 0$ for each $j$ and $\beta_i\geq 0$.
It follows that
\begin{eqnarray} K_{X_{i+1}/X} & =&\sum_{j=1}^i a_j E_j + (\frac{\sum_{j=1}^i a_j \alpha_{i,j} }{r} +\frac{1}{r}) E_{i+1}; \nonumber\\
D_{X_{i+1}/X} & =&\sum_{j=1}^i b_j E_j + (\frac{\sum_{j=1}^i b_j \alpha_{i,j} }{r} +\frac{\beta_i}{r}) E_{i+1}.
\end{eqnarray}

Since $(X, P)$ is Gorenstein,  both $\frac{\sum_{j=1}^i a_j \alpha_{j} }{r} +\frac{1}{r} $ and $\frac{\sum_{j=1}^i b_j \alpha_{j} }{r} +\frac{\beta_i}{r}$ are positive integers.
Hence
$$ \frac{\sum_{j=1}^i a_j \alpha_{i,j}+1 }{r} =\lceil \frac{\sum_{j=1}^i a_j \alpha_{i,j} }{r} \rceil \leq \lceil \frac{\sum_{j=1}^i b_j \alpha_{i,j} }{r} \rceil \le \frac{\sum_{j=1}^i b_j \alpha_{i,j}+\beta_i }{r}.$$
Therefore, $D_{X_{i+1}/X} \ge K_{X_{i+1}/X}$. We are done by
induction.
\end{proof}

%The proof of Theorem \ref{jk} implies that singularities on $Z$ are simpler than those on $X$.  Therefore there %exists a partial resolution with respect to the given divisor $D$ after a finite number of duplications of the step %$\pi_P$, say
%\begin{equation}\pi_D: Z':=Z_{l}\rw \cdots \rw Z_1:=Z\overset{\pi_P}\lrw X\label{Z'}\end{equation}
% where $Z'$ is still projective Gorenstein with at worst terminal singularities and, furthermore, the strict %transform ${\pi_D}_*^{-1}(D)$  does not pass through any singular point of $Z'$.  Thus Theorem \ref{P} directly %implies the following:

% \begin{cor}\label{partial}  For the partial resolution (\ref{Z'}), we have $D_{Z'/X}\geq K_{Z'/X}. $
%\end{cor}

%{\it Assumption ($\star\star$)}.
%{}From now on,

Now, for the given terminal Gorenstein singularity $P \in X$, the feasible resolution $\tau_P$ as in the above Step (3) can be rephrased as:
\begin{equation}\tau_P: Z_l \to Z_{l-1} \to \cdots\to Z_1 \to X \ni P\label{res_P}\end{equation}
by setting $Z_j:=X_{n_j}$, where $Z_l$ is smooth on
$\tau_P^{-1}( P)$ and each birational morphism $Z_i \to Z_{i-1}$ is
a fG partial resolution for all $i$. Therefore Theorem \ref{P} and
 simple induction directly imply the following:

\begin{cor} \label{fGr} For the feasible resolution (\ref{res_P}), we have $D_{Z_j/X}\geq K_{Z_j/X}$ for $1 \le j \le l$.
\end{cor}

In the last part of this section, we focus on moving  linear systems. Suppose that $|M|$ is a moving linear system (i.e. without fixed part) on the given projective Gorenstein terminal 3-fold $X$ with $\text{Bs}|M|\neq \emptyset$.
 %where $M$ is a Cartier divisor.
Similar to usual resolution of indeterminancies, we can have a {\it Gorenstein resolution of indeterminancies} as follows:
\begin{enumerate}
\item[(i)] If $|M|$ is free out of singularities, i.e. $\text{Bs}|M| \cap \text{Sing}(X)= \emptyset$, then we do nothing.

\item[(ii)] If there is a point $P \in  \text{Bs}|M| \cap \text{Sing}(X)$, we take a fG-partial resolution $Z_1 \to X \ni P$ and consider the linear system $|M_1|$, where $M_1$ is the proper transform of $M$ on $Z_1$.

\item[(iii)] Inductively, we will end up with a chain of fG partial resolutions $Z_n \to \ldots \to Z_1 \to X$ so that $|M_n|$ is free out of singularities of $Z_n$ (see (\ref{res_P})), since 3-dimensional terminal singularities are isolated.

\item[(iv)] If $|M_n|$ is base point free on $Z_n$, then we stop. Note that $Z_n$ is a Gorenstein terminal 3-fold.

\item[(v)] If $|M_n|$ has base points, then $\text{Bs}|M_n|$ consists of smooth points of $Z_n$ by our construction. We then consider the usual resolution of indeterminancies over $\text{Bs}|M_n|$, say
$Z_k \to \ldots \to Z_n$, which is composed of a sequence of blow-ups along smooth points or curves by Hironaka's big theorem.

\item[(vi)] Thus we may end up with a 3-fold $Z_k$ so that $|M_k|$
is base point free. We call
\begin{equation}\mu \colon Z_k \stackrel{\tau_k}{\longrightarrow} \ldots \stackrel{\tau_{n+1}}{\longrightarrow} Z_n \stackrel{\tau_n}{\longrightarrow} \ldots \stackrel{\tau_1}{\longrightarrow} X
\label{Gres}\end{equation}
 {\it a Gorenstein resolution of indeterminancies of $|M|$}. Note that $Z_k$ is a Gorenstein terminal 3-fold in general.
\end{enumerate}

 %Take $D$ to be a general member of $|M|$. According to Hironaka's big theorem, Theorem \ref{jk} and the existence %of the partical resolution (\ref{Z'}), we may take the following resolution to the pair $(X,|M|)$:
%\begin{equation}
%\pi: X'\overset{\eta}\lrw Z'\overset{\pi_Z}\lrw X \label{reso}
%\end{equation}
%such that
%\begin{enumerate}
%\item $\pi:X'\rw X$ is a birational morphism so that $\text{Mov}|\pi^*(M)|$ is base point free;
%\item $X'$ is nonsingular projective;
%\item $\pi_Z:Z'\rw X$ is the similar partial resolution to (\ref{Z'})? which resolve only those singular points contained in $\text{Bs}|M|$ and the strict transform ${\pi_Z}_*^{-1}(D)$ does not pass through any singular points of $Z'$;
%\item $\eta$ is the minimal resolution so that Items $(1)\sim(3)$ hold.
%\end{enumerate}

%We divide all the exceptional divisors into two classes:
%\begin{itemize}
%\item[0.] $E^0_i$ ($i=1,\cdots, s$) denotes any one dominating a singular point of $Z'$, i.e. $\eta(E^0_i)\in %\text{Sing}(Z')$;
%\item[I.] $E^I_j$ ($j=1,\cdots, t$) means any of the rest ones.
%\end{itemize}
%Clearly, since $X$ has Gorenstein terminal singularities, we have
%\begin{eqnarray}
%K_{X'}&=&\pi^*K_X+\sum_{j=1}^ta_jE^I_j+\sum_{i=1}^s\hat{a}_iE^0_i\nonumber\\
%\pi^*(D)&=&\pi_*^{-1}(D)+\sum_{j=1}^te_jE^I_j.\label{OK}
%\end{eqnarray}

%\begin{prop}\label{E12} Under the assumption ($\star\star$), we have $a_j\leq 2e_j$ for all $j=1,\cdots, t$.
%\end{prop}
%\begin{proof}
%\end{proof}

\begin{thm}\label{key} Let $|M|$ be a moving linear system on a projective Gorenstein terminal 3-fold $X$ and $D \in |M|$ be a general member.
Let $\mu: Z_k \to X$ be the  Gorenstein resolution of indeterminancies as in (\ref{Gres}). Then $2 D_{Z_k/X} \ge K_{Z_k/X}$.
\end{thm}

\begin{proof} We keep the notation as in above Steps (i)$\sim$ (vi).
For each $i <n$, we have $D_{Z_{i+1} / Z_i} \ge K_{Z_{i+1} / Z_i}$  by Theorem \ref{P}. For each $i \ge n$, $\tau_{i+1}$ is a blowup along a smooth curve or a smooth point, contained in $D_{Z_i}$. Let $E_{i+1}$ be the exceptional divisor.
Then
$2D_{Z_{i+1} / Z_i} \ge 2 E_{i+1} \ge K_{Z_{i+1} / Z_i}$.
Since $D_{Z_{i+1} / X} = \tau_{i+1}^* D_{Z_{i} / X} +D_{Z_{i+1} / Z_i}$ and $K_{Z_{i+1} / X} = \tau_{i+1}^* K_{Z_{i} / X} +K_{Z_{i+1} / Z_i}$. The statement now follows easily by induction.
\end{proof}

%\newpage
\section{\bf The canonical family of curves of genus 2}

Let $X$ be a projective Gorenstein minimal 3-fold of general type.  The fact that $K_X^3$ being even allows us to assume $p_g(X)\geq 5$ in order to prove Theorem \ref{main}. Thus we may always consider the non-trivial canonical map $\varphi_1$.  Set $d:=\dim \overline{\varphi_1(X)}$.

The following inequalities are already known:
\begin{itemize}
\item[I.]  If $d\neq 2$, then
$$K_X^3\geq \text{min}\{2p_g(X)-6, \ \frac{7}{5}p_g(X)-2\}$$
 by \cite[Theorem 5 (1)]{JMSJ} and Catanese--Chen--Zhang \cite[Theorem 4.1]{CCZ}.

\item[II.]  If $d=2$ and $X$ is canonically fibred by curves $C$ of genus $g(C )\geq 3$, then  $K_X^3\geq 2p_g(X)-4$ by \cite[Theorem 4.1(ii)]{JMSJ}.
\end{itemize}

\begin{thm}\label{g2}
Let $X$ be a projective minimal smooth 3-fold of general type.
Suppose that $d=2$ and $X$ is canonically fibred by curves of genus $2$. Then
$$K_X^3\ge \frac{1}{3}(4p_g(X)-10).$$
The inequality is sharp.
\end{thm}
\begin{proof}
Write $|K_X|=|M|+F$, where $|M|$ is the moving part and $F$ is the fixed part. Let
$$\mu: X'=Z_k \to \ldots \to Z_1 \to X$$ be the Gorenstein resolution of indeterminancies as (\ref{Gres}). Let $g=\varphi_1 \circ \mu$
and take the Stein factorization, we have the induced fibration $f \colon X'\lrw W$.

%We keep the same notations as in 1.3 and in Case 1 of the proof of Theorem 4.1.
%Set $K_X\sim \overline{M}+\overline{Z}$, where $\overline{M}$ is the movable part of $|K_X|$ and $\overline{Z}$ is the fixed part.
%We may take the same successive blow-ups
%$$\pi: X'=X_{n}\overset{\pi_n}\to\rw X_{n-1}\rw\cdots\rw
%   X_{i}\overset{\pi_i}\to\rw X_{i-1}\rw\cdots\rw
%  X_1\overset{\pi_1}\to\rw X_0=X$$
%as in the set up for Lemma 4.2.

%Let $g=\phi{1}\circ\pi$. Taking the Stein-factorization of $g$, we get $f:X'\lrw W$.
A general fiber of $f$ is a smooth curve of genus $2$ by assumption of the theorem. Let $D$ be a general member of $|M|$ and $S:=D_{X'}$ be the general member of the moving part of
$|\mu^*M|$. Then
%$$K_{X'}=\pi^*(K_X)+E=\pi^*(K_X)+\sum_{i=0}^pa_iE_i$$
%and $\pi^*(\overline{M})\sim S_1+\sum_{i=0}^pe_iE_i.$
%We know that $a_i\ge 0$, $e_i>0$ and both $a_i$ and $e_i$ are integers for all $i$.
we  have
%\begin{eqnarray*}
$$\mu^*K_X=\mu^*{M}+\mu^*F=S+D_{X'/X}+\mu^*F.$$
%&\sim S+\sum_{i=0}^pe_i'E_i+\sum_{j=1}^q d_jL_j
%=S+E',\end{eqnarray*}
%where $e_i'\ge e_i$.
Set $E':=D_{X'/X}+\mu^*F$.

On the surface $S$, set $L:=\mu^*(K_X)|_{S}$. We also have
$S|_{S}\equiv aC$ where $a\ge p_g(X)-2$ and $C$ is a general fiber of the restricted fibration
$f|_{S}: S\lrw f(S).$
Note that the above $C$ lies in the same numerical class as that of a general fiber of $f$.
 One has
 $$ (\mu^*K_X^2 \cdot S) \ge (\mu^*K_X \cdot_S S) \ge a (L \cdot C)  \ge (L\cdot C)(p_g(X)-2).$$
If $(L \cdot C) \ge 2$, then we have already $K_X^3 \ge (\mu^*K_X^2 \cdot S) \ge 2p_g(X)-4$.
% we have already seen in the proof of Theorem 4.1 that
%$K_X^3\ge 2p_g(X)-4.$
%{}From now on,
It remains to consider the case $(L\cdot C)=1$. Note that, in this situation, $|M|$ must have base points. Otherwise, $\mu=\text{identity}$ and
$$(L\cdot C)=(K_X|_{S}\cdot C)=((K_X+S)|_{S}\cdot C)=(K_{S}\cdot C)=2,$$
a contradiction.

Denote $E'|_{S}:=E_V'+E_H'$, where $E_V'$ is the vertical part and $E_H'$ is the horizontal part with respect to $f|_S$.
Since $(E'_H \cdot C)=(E'|_{S}\cdot C)=(L\cdot C)=1$,  $E_H'$ is an irreducible curve and is a section of the restricted fibration $f|_{S}$.

Denote $K_{X'/X}|_{S}:=E_V+E_H$, where $E_V$ is the vertical part and $E_H$ is the horizontal part. {}From $(K_{S}\cdot C)=2$, one sees that $(K_{X'/X}|_S \cdot C)=1$ and hence
$(E_H\cdot C)=1$. This also means that $E_H$ is an irreducible curve and we may assume that $E_H=E_0|_S$ for some $\mu$-exceptional divisor $E_0$.
Notice that $2E' \ge K_{X'/X}$ by Theorem \ref{key}. In particular $E_0$ is contained in $E'$. Therefore, $E'_H=E_H$ and  $2E'_V \ge E_V$.

% Because $e_0'>0$ and $\pi^*(K_X)\cdot C=1$, we see that $e_0'=1$ and thus $E_H'$ also comes from $E_0$. Since %$E_0|_{S_1}$ has only one horizontal part, $E_H$ and $E_H'$ coincide with a curve $G$. Now It is quite clear that
%$$E_V=\sum_{i=1}^pa_i(E_i|_{S_1})+(E_0|_{S_1}-G),$$
%$$E_V'=\sum_{i=1}^pe_i'(E_i|_{S_1})+\sum_{j=1}^qd_j(L_j|_{S_1})
%+(E_0|_{S_1}-G).$$
%We have the following
%\medskip

%\noindent{\bf Claim.}\ \ $E_V\le 2E_V'.$
%\medskip

%This is apparently a direct consequence of Lemma 4.2. In fact, we have $a_i\le 2e_i\le 2e_i'$ by Lemma 4.2 for all %$i>0$.
%Thus
%$$\sum_{i=1}^pa_i(E_i|_{S_1})\le 2\sum_{i=1}^pe_i'(E_i|_{S_1})
%\le 2(\sum_{i=1}^pe_i'(E_i|_{S_1})+\sum_{j=1}^qd_j(L_j|_{S_1})).$$
%On the other hand, It is obvious that
%$E_0|_{S_1}-G\le 2(E_0|_{S_1}-G).$
%Therefore we get
%\begin{eqnarray*}
%E_V&=(E_0|_{S_1}-G)+\sum_{i=1}^pa_i(E_i|_{S_1})\\
%&\le 2(E_0|_{S_1}-G)
%+2(\sum_{i=1}^pe_i'(E_i|_{S_1})+\sum_{j=1}^qd_j(L_j|_{S_1}))
%=2E_V'
%\end{eqnarray*}
%and the claim is proved.
Let $G:=E_H=E'_H$.
Since $2E_V'-E_V$ is effective and vertical, we see that
$ 2(E_V'\cdot G)  \ge (E_V\cdot G) $. On the surface $S$, we have
%$$(K_{S}+G) \cdot G=2p_a(G)-2 \ge -2.$$
%On the other hand, we have
$$\begin{array}{ll}
& (2\mu^*K_X|_S +E_V') \cdot G\\
=& (\mu^*K_X|_S+S|_S+2E_V'+E_H') \cdot G \\
\ge & (\mu^*K_X|_S+S|_S+ E_V+E_H) \cdot G \\
=& (\mu^*K_X+K_{X'/X}|_S+S|_S) \cdot G\\
=& (K_S \cdot G) \ge -2-G^2
\end{array}$$
%So we have
%$$2\pi^*(K_X)|_{S_1}\cdot G+E_V'\cdot G+G^2+2\ge 0. $$
We also have
$$\begin{array}{ll} (\mu^*K_X|_{S}-E_V' )\cdot G &=(S|_{S}\cdot G)+ (E_H' \cdot G) \\
                                            &=a(C \cdot G) + G^2 \\
                                             & \ge p_g(X)-2 +G^2. \end{array}$$
Combining these, we get $3(\mu^*(K_X)|_{S}\cdot G)\ge p_g(X)-4$ and therefore
$$(\mu^*(K_X)|_S \cdot E'|_S) \ge
(\mu^*(K_X)|_{S}\cdot G)\ge\frac{1}{3}(p_g(X)-4).$$
{}Finally we have
$$
\begin{array}{ll}
K_X^3&=\mu^*(K_X)^3\ge (\mu^*(K_X)^2\cdot S)\\
&= (\mu^*(K_X)|_{S}\cdot S|_{S})+(\mu^*(K_X)|_{S}\cdot E'|_{S})\\
&\ge (p_g(X)-2)+\frac{1}{3}(p_g(X)-4)=\frac{2}{3}(2p_g(X)-5).
\end{array}
$$
The inequality is sharp by virtue of Kobayashi's example \cite{Kob}.
\end{proof}

Theorem \ref{main} follows directly from known results I, II and Theorem \ref{g2}.

We would like to ask the following:
\begin{op} Is the inequality $K_X^3\geq \frac{4}{3}p_g(X)-\frac{10}{3}$ true for any projective minimal 3-fold $X$ of general type?
\end{op}

Some known results includes:  if $p_g(X) \ge 3$, then $K_X^3 \ge
1$ and if $p_g(X) \ge 4$, then $K_X^3 \ge 2$ (cf. \cite[Theorem
1.5]{MAnn}).

\end{document}